\documentclass[11pt,english,a4paper]{amsart}
\pdfoutput=1

\usepackage[utf8]{inputenc}
\usepackage[T1]{fontenc}
\usepackage{lmodern}
\usepackage{babel}
\usepackage{amsmath,amssymb,amsthm}
\usepackage[marginratio={1:1,1:1},totalheight=600pt]{geometry}
\usepackage{microtype}
\usepackage{hyperref}
\usepackage{url}
\usepackage{tikz-cd}
\usepackage{colortbl}
\usepackage[normalem]{ulem}
\usepackage{rotating}

\theoremstyle{plain}
\newtheorem{theorem}{Theorem}[section]
\newtheorem{lemma}[theorem]{Lemma}
\newtheorem{proposition}[theorem]{Proposition}
\newtheorem{corollary}[theorem]{Corollary}

\theoremstyle{definition}
\newtheorem{definition}[theorem]{Definition}
\newtheorem{example}[theorem]{Example}
\newtheorem{remark}[theorem]{Remark}
\newtheorem{question}[theorem]{Question}

\theoremstyle{remark}

\newcommand{\Z}{\mathbb{Z}}

\newcommand{\F}{\mathbb{F}}

\title[Shift-invariant almost liftings]{Shift-invariant transformations and almost liftings}

\author[Haugland]{Jan Kristian Haugland}
\address{Norwegian National Security Authority (NSM)\\Norway}
\email{admin@neutreeko.net}

\author[Omland]{Tron Omland}
\address{Norwegian National Security Authority (NSM) \and Department of Mathematics, University of Oslo\\Norway}
\email{tron.omland@gmail.com}

\date{November 3, 2025}

\begin{document}
	
	\begin{abstract} 
		We investigate shift-invariant transformations, also known as rotation-symmetric vectorial Boolean functions, on $n$~bits that are induced from Boolean functions on $k$~bits, for $k\leq n$. We consider such transformations that are not necessarily permutations, but are, in some sense, almost bijective, and study their cryptographic properties.
        
        In this context, we define an almost lifting as a Boolean function for which there is an upper bound on the number of collisions of its induced transformation that does not depend on $n$. We show that if a Boolean function with diameter $k$ is an almost lifting, then the maximum number of collisions of its induced transformation is $2^{k-1}$ for any $n$.
		
		Moreover, we search for functions in the class of almost liftings that have good cryptographic properties and for which the non-bijectivity does not cause major security weaknesses.
		
		These functions generalize the well-known map $\chi$ used in the Keccak hash function.
	\end{abstract}
	
	\maketitle
	
	\section{Introduction}
	
	In symmetric cryptography, the ciphers often consist of linear and nonlinear operations in layers, where the nonlinear part is determined by a so-called S-box, short for ``substitution box'', which is a permutation on the set $\F_2^n$ of $n$-bit vectors. All the substitution-permutation networks are of this type, including the current block cipher standard, AES, and the S-boxes are fundamental in increasing confusion and diffusion to such ciphers.
	Moreover, lookup tables typically have large implementation costs, so good candidates for S-boxes are permutations with an easy description and good cryptographic properties. Shift-invariant bijections (also known as rotation-symmetric permutations) have been shown to be useful in this context, e.g., in lightweight cryptography. The potential of symmetric structure is already illustrated by the $\chi$ function used in Keccak.
	
	In this paper, we relax the bijectivity requirement on the nonlinear layer and allow some collisions. In particular, we study ``near-permutation S-boxes'' that are ``almost bijective'' shift-invariant transformations $\F_2^n \to \F_2^n$ induced from Boolean functions. This is motivated by use cases where the inverse is not needed, e.g., certain modes of operation of a block cipher, constructions such as CBC-MAC and other pseudorandom functions, and then especially where symmetry simplifies hardware or software implementation. In such cases, one needs to investigate whether collisions due to non-invertibility form a threat to security.   A systematic study and exploration of almost bijective shift-invariant S-boxes with compact rules will allow greater flexibility in cryptographic design.
    To pursue this approach, we first need to discuss what ``almost bijective'' should mean, e.g., two natural properties to demand are that the image should be large and that the number of collisions should be small.

    Henceforth, we will use the term S-box also for functions $\F_2^n \to \F_2^n$ that are not necessarily permutations.
	Let $F\colon\F_2^n\to\F_2^n$ be an S-box and $\sigma$ be the right shift, that is, $\sigma(x_1,x_2,\ldots,x_n)=(x_n,x_1,\ldots,x_{n-1})$.
	Then $F$ is shift-invariant (or rotation-symmetric) if $F\circ \sigma=\sigma\circ F$, and $F$ is then completely determined by a Boolean function $f\colon\F_2^n\to\F_2$. 
    Conversely, any Boolean function $f$ on $k$~bits determines a shift-invariant S-box $F$ on $n$~bits, for $n\geq k$, by
 \[
F(x_1,x_2,\dotsc,x_{n})=\big(f(x_1,x_2,\dotsc,x_{k}),f(x_2,x_3,\dotsc,x_{k+1}), \dotsc,f(x_n,x_1,\dotsc,x_{k-1})\big).
\]
    The compact description implies that shift-invariant S-boxes with sufficiently good cryptographic properties are candidates to be used as primitives in implementation-efficient symmetric ciphers.
	The motivating example is the function $\chi(x_1,x_2,x_3)=x_1\oplus (1\oplus x_2)x_3$, first studied in Daemen's thesis~\cite{JDA-thesis}. The function $\chi$ gives rise to bijections for all odd $n\geq 3$ with good cryptographic properties and is used in the hash function Keccak \cite{keccak}, but it is also interesting to look at the non-bijective case, for even $n$.
	
	Examples of good cryptographic properties are: no differentials with high differential probability, no linear approximations with high linear potential. For implementation, we want low computational complexity and as much symmetry as we can get.
	A low algebraic degree is good for protection against side-channel attacks by means of masking, while a high algebraic degree is good for protection against higher order differential attacks. A dense algebraic normal form protects better against integral attacks, but relatively sparse ones can be compensated for by taking a linear layer with large diffusion. Moreover, some desirable properties for almost bijectivity could be:
 	\begin{itemize}
	\item[(P1)] the number of collisions $\max_y \lvert F^{-1}(y)\rvert$ should be low,
	\item[(P2)] (size of the image of $F$)/(size of the codomain of $F$) should be high,
	\item[(P3)] the image $\operatorname{Im}(F)$ and its complement should be unstructured in $\F_2^n$.
	\end{itemize}

	More concretely, we search for Boolean functions on up to five bits with simple descriptions that induce S-boxes with good cryptographic properties. We believe that non-bijective shift-invariant S-boxes will be useful in the applications mentioned above. 
	
	Shift-invariant S-boxes can be extended to arbitrarily large dimensions and viewed as cellular automata, which are certain dynamical systems on the space of infinite binary strings indexed by $\Z$, thought of as cells, where the state of a cell at the next time step is determined by an update rule depending on a finite number of neighboring cells and uniformly applied to all cells at the same time, see e.g.~\cite{Mariot-24,mariot2021evolutionary, mariot2019cellular}.
	Cellular automata that are reversible correspond to shift-invariant S-boxes that are bijective in all dimensions, so the almost liftings we consider in this paper correspond to ``almost reversible'' cellular automata, which actually coincide with those that are surjective \cite{hedlund}. These are less studied, but still have applications in physics and biology, typically for simulation of microsystems that exhibit non-equilibrium behavior and history-dependent dynamics.
	
	Even though shift-invariant S-boxes (or cellular automata) can be described by simple rules, finding the ones that are bijective is difficult, but previous works and computational data indicate that there are still a lot of examples (see e.g., \cite[Appendix~A]{JDA-thesis} and \cite{OS-iacr}). The same applies to the almost bijective case, where few families are described.
	
	In this paper, for a shift-invariant function $F$ induced from a Boolean function $f$, as described above, we first discuss when $f$ is what we call a potential lifting in Section~\ref{sec-potential}. The purpose is to reduce the search space, when looking for functions with desirable properties. We provide some tables in the appendix for the number of such functions, which is also helpful when trying to find proper liftings, that is, functions for which the induced $F$ is bijective.
	
	Further, in Section~\ref{sec-almost}, we introduce almost liftings as Boolean functions for which there is an upper bound for the number of collisions of its induced functions that does not depend on $n$. We then prove Theorem~\ref{main-theorem}, stating that if a Boolean function with diameter $k$ is an almost lifting, then the maximum number of collisions of its induced functions is $2^{k-1}$ for any $n$. This means that all functions we consider will satisfy property (P1) in the above list, or at least that we have some control of the number of collisions.
	
	Our Proposition~\ref{virtual} combined with computer experiments provides a conjecture for what the best possible values for (P2) are. The Boolean functions giving rise to these values will be called virtual liftings and we give a complete list of such functions for $k\leq 5$. Property (P3) may be hard to achieve, and in practice, it can be taken care of by carefully designing the linear layer.
	
	In Section~\ref{sec-selection} we choose a selection of functions that are potentially applicable in symmetric ciphers, and compute various cryptographic properties for these functions. It is not clear that our selection is the best one, and there are probably other properties that come into play as well. In other words, there is more investigation left for future work.

    \subsection*{Our contribution}

We introduce a class of Boolean functions, called almost liftings, for use in design of symmetric primitives where low implementation cost is prioritized over full invertibility. Almost liftings define transformations on $n$~bits defined by a local rule of diameter~$k$ with few collisions and provide the framework for studying the use of rotation-symmetry in ``near-permutation-based cryptography''. We formalize this concept and give a precise description of the class of such functions and also subclasses of specific interest, and connect our findings to results from cellular automata theory. We prove several theoretical results about almost liftings and then run computer experiments to find functions with good cryptographic properties.

    \subsection*{Terminology}

A function $F\colon\F_2^n\to\F_2^n$ that commutes with the right shift, i.e., $F\circ\sigma=\sigma\circ F$ is often called ``rotation-symmetric'' in cryptography. In this paper, we instead use the word ``shift-invariant'', inspired by \cite[Section~6]{JDA-thesis}, matching the tradition from cellular automata, and also used in cryptography, e.g., \cite[Section~2.1]{mariot2021evolutionary}. Moreover, for maps $X\to X$ we use the words ``transformation'' and ``function'' interchangeably, and the same goes for ``permutation'' and ``bijection''. Finally, we remark that the Boolean functions $f$ used in these constructions are themselves not rotation-symmetric.

\subsection*{Acknowledgements}
The authors would like to thank Sondre Rønjom for suggesting this direction of research, and Joan Daemen for email communication that has been of great help, and finally to the anonymous referees for useful comments on the manuscript. The second author was partially supported by the Research Council of Norway (RCN), project no. 345433.

	\section{Potential liftings}\label{sec-potential}
	
	Let $f\colon\F_2^k\to\F_2$ be a Boolean function.
	The diameter of $f$ is the length of the consecutive input sequence that the values of $f$ depend on. If $1\leq i\leq j\leq k$ are such that $i$ and $j$ is the smallest and largest number, respectively, such that $f$ depends on $x_i$ and $x_j$, then its diameter is $j-i+1$. If $f$ depends on both $x_1$ and $x_k$, then its diameter is $k$.
	
	For every $n\geq k$ we say that $f$ is a proper $(k,n)$-lifting if the diameter of $f$ is $k$ and $F\colon \F_2^n\to\F_2^n$ defined by
	\[
	F(x_1,\dotsc,x_n)= \Big( f(x_1,x_2,\dotsc,x_k),
	f(x_2,x_3,\dotsc,x_{k+1}),\dotsc,
	f(x_n,x_1,\dotsc, x_{k-1}) \Big)
	\]
	is a bijection. Note the discrepancy between this definition and the one from \cite{OS-iacr}, where it is not required that the diameter is equal to $k$. The reason for assuming full diameter is only a matter of presentation. All of the arguments hold also without this requirement.

	\begin{question}
		Problems concerning bijectivity of the induced functions are generally hard, and although it will not be the main focus of this paper, we list a few of them:
  \begin{itemize}
	\item[(i)] For a given $f\colon\F_2^k\to\F_2$, find the set $\{n\geq k \mid F\colon\F_2^n\to\F_2^n \text{ is bijective} \}$.
	\item[(ii)] For a given a pair $(k,n)$, find all $f\colon\F_2^k\to\F_2$ that induce bijections $F\colon\F_2^n\to\F_2^n$.
	\item[(iii)] Find all functions $f\colon\F_2^k\to\F_2$ that induce bijections $F\colon\F_2^n\to\F_2^n$ for every $n\geq k$.
 \end{itemize}
\end{question}
 
	For every $m\geq k$ define $F_{(m)}\colon\F_2^m\to\F_2^{m-k+1}$ to be the induced function of $f$ that does not wrap around, i.e., with nonperiodic boundary conditions, that is,
	\[
	F_{(m)}(x_1,\dotsc,x_m)= \Big( f(x_1,x_2,\dotsc,x_k),
	f(x_2,\dotsc,x_{k+1}),\dotsc,
	f(x_{m-k+1},\dotsc, x_m) \Big).
	\]
	As usual, we say that $F_{(m)}$ is \emph{balanced} if for all $y\in\F_2^{m-k+1}$
	\[
	\lvert F_{(m)}^{-1}(y)\rvert = 2^{k-1}.
	\]
	
	\begin{lemma}\label{uniform-distribution}
		If $f$ is a $(k,n)$-lifting then $F_{(m)}$ is balanced whenever $k\leq m\leq n$.
	\end{lemma}
	
	\begin{proof}
		Let $m\geq k$, pick $y\in\F_2^{m-k+1}$, and set
		\[
		Y=\{ z\in\F_2^n : z=(y,y') \text{ for some } y'\in\F_2^{n-(m-k+1)} \}.
		\]
		Then $F_{(m)}(x)=y$ if and only if $F(x,x')\in Y$ for every $x'\in\F_2^{n-m}$, so
		\[
		\lvert F_{(m)}^{-1}(y)\rvert=\frac{\lvert F^{-1}(Y)\rvert}{2^{n-m}}=\frac{\lvert Y\rvert}{2^{n-m}}=\frac{2^{n-(m-k+1)}}{2^{n-m}}=2^{k-1},
		\]
		where the second equality follows by bijectivity of $F$.
	\end{proof}
	
	\begin{definition}
		A Boolean function $f\colon\F_2^k\to\F_2$ of diameter $k$ is called a \emph{potential $(k,n)$-lifting} if $F_{(m)}$ is balanced for every $m$ such that $k\leq m\leq n$.
	\end{definition}
	
	\begin{corollary}
		If $k\leq n\leq n'$ and $f$ is a potential $(k,n')$-lifting, then $f$ is also a potential $(k,n)$-lifting.
		
		If $k\leq m\leq m'$ and $F_{(m')}$ is balanced, then $F_{(m)}$ is balanced.
	\end{corollary}
	
	\begin{proof}
		The first statement follows directly from the definition.
		For the latter statement, let $k\leq m\leq m'$, pick $y\in\F_2^{m-k+1}$, and set
		\[
		Y=\{ z\in\F_2^{m'-k+1} : z=(y,y') \text{ for some } y'\in\F_2^{m'-m} \}.
		\]
		Since $\lvert F_{(m')}^{-1}(Y)\rvert=2^{m'-m}\lvert F_{(m)}^{-1}(y)\rvert$ and $F_{(m')}$ is balanced, we get
		\[
		\lvert F_{(m)}^{-1}(y)\rvert=\frac{\lvert F_{(m')}^{-1}(Y)\rvert}{2^{m'-m}}=\frac{2^{k-1}\lvert Y\rvert}{2^{m'-m}}=2^{k-1}.\qedhere
		\]
	\end{proof}
	
	\begin{remark}
		It is observed in \cite{OS-iacr} that $f$ can only be a proper $(k, n)$-lifting if $f(0, 0, \dotsc , 0) \neq f(1, 1, \dotsc , 1)$, but this is not required for potential $(k, n)$-liftings. However, when searching for proper $(k, n)$-liftings, to reduce the space, it would still be natural to consider only the potential $(k, n)$-liftings satisfying $f(0, \dotsc , 0) = 0$ and $f(1, \dotsc , 1) = 1$.
	\end{remark}
	
	\begin{remark}
		It follows from the definition that all potential $(k,n)$-liftings must be balanced, and a balanced Boolean function in $k$ variables cannot have algebraic degree $k$ (see \cite{cusick}, Theorem~2.5). Therefore, all potential $(k,n)$-liftings have degree at most $k-1$.
  \end{remark}

  \begin{proposition}
  If $f$ is a potential $(k,n)$-lifting for $n\geq 2k-1$, then $f$ is balanced on either the subspace $x_1 = 0$ or the subspace $x_k=0$.
    \end{proposition}
  \begin{proof}
  First, $f(x_1, \dotsc, x_k) \oplus f(x_k, \dotsc, x_{2k-1})$ must be balanced (this is a special case of \cite[Proposition~35]{carlet-book}). Let $e_{\alpha \beta}$ denote the number of vectors $x \in \F_2^k$ for which $x_1=\alpha$, $x_k=\beta$ and $f(x)=0$, minus the ``expected'' value $2^{k-3}$. Then, the number of vectors $x \in \F_2^{2k-1}$ for which $x_1 = \alpha$, $x_k = \beta$, $x_{2k-1}=\gamma$ and $f(x_1, \dotsc, x_k) \oplus f(x_k, \dotsc, x_{2k-1}) = 0$ minus the expected value is given by $$\left( 2^{k-3} + e_{\alpha \beta} \right) \left( 2^{k-3} + e_{\beta \gamma} \right) + \left( 2^{k-3} - e_{\alpha \beta} \right) \left( 2^{k-3} - e_{\beta \gamma} \right) - 2^{2k-5} = 2e_{\alpha \beta} e_{\beta \gamma}.$$ Counting over all possibilities of $\{ \alpha, \beta, \gamma \}$, it follows that $$e_{00} e_{00} + e_{00} e_{01} + e_{01} e_{10} + e_{01} e_{11} + e_{10} e_{00} + e_{10} e_{01} + e_{11} e_{10} + e_{11} e_{11} = 0$$ which can be written as $(e_{00} + e_{10})(e_{00} + e_{01}) + (e_{01} + e_{11})(e_{10} + e{11}) = 0$. Since $e_{00} + e_{01} + e_{10} + e_{11} = 0$, the two main terms are identical, and we thus have $(e_{00} + e_{10})(e_{00} + e_{01}) = 0$, or equivalently, that $f$ is balanced on either the subspace $x_1 = 0$ or the subspace $x_k=0$.
	\end{proof}
	
	Let $S_{k, n}$ denote the set of all $f\colon \F_2^k\to\F_2$ such that $f$ is a potential $(k,n)$-lifting and $f(0, 0, \dotsc , 0) = 0$, and let $S_k = \{ f\colon \F_2^k\to\F_2 \mid f \in S_{k, n} \text{ for all } n\geq k\}$. Data suggest that we have $\lvert S_3\rvert=10$, $\lvert S_4\rvert=264$, and $\rvert S_5\rvert=70942$. Among these functions, $5$, $132$, and $35450$, respectively, satisfy $f(1, \dotsc, 1)=1$.

\medskip
    A Boolean function $f\colon\F_2^k\to\F_2$ is called \emph{permutive} if $f(x)\oplus x_1$ is independent of $x_1$ or if $f(x)\oplus x_k$ is independent of $x_k$. We will now see that if $f$ is permutive and has diameter $k$, then it is an almost lifting.
 
	\begin{proposition}\label{permutive}
		For any two Boolean functions $h, h'\colon\F_2^{k-1} \to \F_2$, where $h(x)$ depends on $x_1$ and $h'(x)$ depends on $x_{k-1}$, the permutive functions $f, g\colon\F_2^k \to \F_2$ given by $f(x_1, \dotsc, x_k) = h(x_1, \dotsc, x_{k-1}) \oplus x_k$ and $g(x_1, \dotsc, x_k) = x_1 \oplus h'(x_2, \dotsc, x_k)$ are potential $(k, n)$-liftings for all $n \geq k$.
	\end{proposition}
	
	\begin{proof}
		Suppose $f$ has this form, and take any $y \in \F_2^{m-k+1}$. The diameter of $f$ is clearly $k$, and it suffices to prove that for any $z\in \F_2^{k-1}$, there is exactly one element of the form $x=(z,w) \in  F_{(m)}^{-1}(y)$. Indeed, there are $2^{k-1}$ elements in $\F_2^{k-1}$, so this would give that $\lvert F_{(m)}^{-1}(y)\rvert=2^{k-1}$. But given $x_1, \dotsc, x_{k-1}$ for some element $x \in F_{(m)}^{-1}(y)$, and for $i=0, 1, \dotsc, m-k$ in turn, we necessarily have $x_{k + i} = y_{i + 1} \oplus h(x_{i + 1}, \dotsc, x_{i + k - 1)}$. The corresponding argument for $g$ is immediate by symmetry.
	\end{proof}

	\begin{corollary}
		We have that $\lvert S_k\rvert \geq 2^{2^{k-1}}-3 \cdot 2^{2^{k-2}-1}$.
	\end{corollary}
	
	\begin{proof}
		The number of functions of the form $h(x_1, \dotsc, x_{k-1}) \oplus x_k$ that depend on $x_1$ is $2^{2^{k-1}} - 2^{2^{k-2}}$, and similarly for functions of the form $x_1 \oplus h'(x_2, \dotsc, x_k)$ that depend on $x_k$. There are $2^{2^{k-2}}$ functions in the intersection, i.e., of the form $x_1 \oplus h(x_2, \dotsc,  x_{k-1}) \oplus x_k$. This gives us $2 \cdot 2^{2^{k-1}} - 3 \cdot 2^{2^{k-2}}$ distinct functions, of which half satisfy $f(0,\dotsc, 0) = 0$.
	\end{proof}
	
	The corollary gives us $\lvert S_3\rvert \geq 10$, $\lvert S_4\rvert \geq 232$ and $\lvert S_5\rvert \geq 65152$, which is not far from the actual values.

	\begin{definition}
		Consider the maps $c,r\colon \F_2^k\to\F_2^k$ given by complementing and reflecting, that commute, and are defined by
		\[
		c(x_1,x_2,\dotsc,x_k)=(\overline{x_1},\overline{x_2},\dotsc,\overline{x_k}) \quad\text{and}\quad
		r(x_1,x_2,\dotsc,x_k)=(x_k,\dotsc,x_2,x_1).
		\]
		We say that two Boolean functions $f,g\colon\F_2^k\to\F_2$ are \emph{elementary equivalent} if there are $i,j,\ell\in\{0,1\}$ such that
		\[
		g(x)\oplus \ell=f\circ r^i\circ c^j(x).
		\]
		There are at most eight functions in such an equivalence class, and the corresponding induced functions have identical cryptographic properties.
	\end{definition}

    \begin{corollary}\label{permutive-bound}
    The number of elementary equivalence classes of permutive Boolean functions of diameter $k$ is equal to
    \[
    \frac{1}{8}
    \begin{cases}
        2 \cdot 2^{2^{k-1}} + 2^{2^{k-2}} + 2 \cdot 2^{2^{k-3} + 2^{\frac{k}{2}-2}} - 6 \cdot 2^{\lfloor 2^{k-3} \rfloor} &\quad\text{if }k \equiv 0\, (\operatorname{mod} 2)\\
        2 \cdot 2^{2^{k-1}} + 2^{2^{k-2}} + 2^{2^{k-3} + 2^{\frac{k-3}{2}}} - 4 \cdot 2^{2^{k-3}} &\quad\text{if }k \equiv 1\, (\operatorname{mod} 2)
    \end{cases}
    \]
    \end{corollary}

    \begin{proof}
    Let $T_k$ denote the set of potential $(k, n)$-liftings given by Lemma~\ref{permutive}. Let $T^r_k = \{ f \in T_k : f(x) = f \circ r(x) \oplus \kappa\}$ for some $\kappa$, and similarly $T^c_k = \{ f \in T_k : f(x) = f \circ c(x) \oplus \kappa\}$ and $T^{r \circ c}_k = \{ f \in T_k : f(x) = f \circ r \circ c(x) \oplus \kappa\}$ for some $\kappa$. The expression $\lvert T_k\rvert + \lvert T^r_k\rvert + \lvert T^c_k\rvert + \lvert T^{r \circ c}_k\rvert$ counts each function with 8 elements in its equivalence class once, each function with 4 elements in its equivalence class twice, and each function with 2 elements in its equivalence class four times. Thus, the number of equivalence classes of $T_k$ is $$\frac{\lvert T_k\rvert + \lvert T^r_k\rvert + \lvert T^c_k\rvert + \lvert T^{r \circ c}_k\rvert}{8}.$$
    
    The number of functions of the form $h(x_1, \dotsc, x_{k-1}) \oplus x_k$ that depend on $x_1$ is $2^{2^{k-1}} - 2^{2^{k-2}}$, and similarly for functions of the form $x_1 \oplus h'(x_2, \dotsc, x_k)$ that depend on $x_k$. There are $2^{2^{k-2}}$ functions in the intersection, i.e., of the form $x_1 \oplus h(x_2, \dotsc,  x_{k-1}) \oplus x_k$. Thus, $\lvert T_k\rvert = 2 \cdot 2^{2^{k-1}} - 3 \cdot 2^{2^{k-2}}$.
    
    The functions in $T^r_k$ must have $\kappa=0$ since $r$ fixes some inputs, and have the form $f(x) = x_1 \oplus h(x_2, \dotsc, x_{k-1}) \oplus x_k$. The expression $2^{k-2} + 2^{\lfloor \frac{k-1}{2}\rfloor}$ counts the $x$'s for which $h(x) = h \circ r(x)$ twice and those for which $h(x) \neq h \circ r(x)$ once. Thus,
    \[
    \lvert T^r_k\rvert=
    \begin{cases}
        2^{\frac{2^{k-2} + 2^{\frac{k}{2}-1}}{2}} &\quad\text{if }k \equiv 0\, (\operatorname{mod} 2)\\
        2^{\frac{2^{k-2} + 2^{\frac{k-1}{2}}}{2}} &\quad\text{if }k \equiv 1\, (\operatorname{mod} 2)
    \end{cases}
    \]
    The functions in $T^c_k$ could have $\kappa$ equal to either 0 or 1, except for $k=2$, in which case we must have $f(x) = x_1 \oplus x_2 \oplus $ some constant which implies $\kappa=0$. Like in the derivation of $\lvert T_k\rvert$ above, we have $$\lvert T^c_k\rvert = 2 \left( 2 \cdot 2^{2^{k-2}} - 3 \cdot 2^{\lfloor 2^{k-3} \rfloor} \right)$$ (which happens to give the correct number also for $k=2$).
    
    The functions in $T^{r \circ c}_k$ must have $\kappa=0$ if $k$ is even, otherwise it can be either 0 or 1, and they must have the form $f(x) = x_1 \oplus h(x_2, \dotsc, x_{k-1}) \oplus x_k$. Thus,
    \[
    \lvert T^{r \circ c}_k\rvert=
    \begin{cases}
        2^{\frac{2^{k-2}+2^{\frac{k}{2}-1}}{2}} &\quad\text{if }k \equiv 0\, (\operatorname{mod} 2)\\
        2 \cdot 2^{\frac{2^{k-2}}{2}} &\quad\text{if }k \equiv 1\, (\operatorname{mod} 2)
    \end{cases}
    \]
    Thus, the number of classes is
    \[
    \frac{1}{8}
    \begin{cases}
        2 \cdot 2^{2^{k-1}} + 2^{2^{k-2}} + 2 \cdot 2^{2^{k-3} + 2^{\frac{k}{2}-2}} - 6 \cdot 2^{\lfloor 2^{k-3} \rfloor} &\quad\text{if }k \equiv 0\, (\operatorname{mod} 2)\\
        2 \cdot 2^{2^{k-1}} + 2^{2^{k-2}} + 2^{2^{k-3} + 2^{\frac{k-3}{2}}} - 4 \cdot 2^{2^{k-3}} &\quad\text{if }k \equiv 1\, (\operatorname{mod} 2)
    \end{cases}
    \]
    \end{proof}

    Comparison between the number of elementary equivalence classes of all almost liftings and of all permutive Boolean functions of diameter $k=3,4,5$, using Corollary~\ref{permutive-bound}:
    \[
	\begin{array}{|c|c|c|} \hline
		k & \#\,\text{almost liftings} & \#\,\text{permutive} \\ \hline
		3 & 4 & 4 \\
		4 & 73 & 65 \\
		5 & 17881 & 16416 \\ \hline
	\end{array}
	\]
	
	Detailed tables for $k=3,4,5$ are given in Appendix~\ref{liftings}.
	
	\section{Almost liftings}\label{sec-almost}
	
	Let $f\colon\F_2^k\to\F_2$ be a Boolean function and for every $n\geq k$ we define $F\colon \F_2^n\to\F_2^n$ by
	\[
	F(x_1,\dotsc,x_n)= \Big( f(x_1,x_2,\dotsc,x_k),
	f(x_2,x_3,\dotsc,x_{k+1}),\dotsc,
	f(x_{n-k+1},x_{n-k+2},\dotsc, x_n) \Big),
	\]
	and set
	\[
	\ell_n(f)=\max_{y\in\F_2^n} \lvert F^{-1}(y)\rvert,
	\]
	\[
	\ell(f)=\sup_{n\geq k}\ell_n(f).
	\]
	
	\begin{definition}
		Let $f\colon\F_2^k\to\F_2$ be a Boolean function of diameter $k$.
		If $\ell(f)<\infty$, we say that $f$ is an \emph{almost lifting}, and if $\ell(f)=1$ we say that $f$ is a \emph{proper lifting}.
  \end{definition}

\begin{example}
\begin{itemize}
    \item[(i)]
    Let $\chi(x)=x_1\oplus (x_2\oplus 1)x_3$ be the function used in, e.g., Keccak. Then it is known that $\ell_n(\chi)=1$ for $n$ odd and $3$ for $n$ even. In particular, $\sup_{n\geq k}\ell_n(\chi)=3$ so $\chi$ is an almost lifting.
    \item[(ii)]
    The function $f(x)=x_2\oplus x_1(x_3\oplus1)x_4$ is, up to elementary equivalence, the only nonaffine proper lifting for $k\leq 4$ (first observed by Patt \cite{patt}), i.e., $\ell_n(f)=1$ for all $n$.
    \item[(iii)]
    Define the function $g(x)=x_1\oplus x_2(x_3\oplus x_4\oplus 1)$. Then $f$ is an almost lifting, since it is permutive, but it seems like $\ell_n(f)$ is a nonperiodic sequence, that we computed for $4\leq n\leq 20$ to be $4,2,4,2,4,2,3,2,4,2,3,3,4,2,4,3,4$.
\end{itemize}

\end{example}

    A proper lifting is called a ``locally invertible'' function in \cite{JDA-thesis}, while a function is called ``globally invertible over $n$'' in \cite{JDA-thesis} if it is a proper $(k,n)$-lifting.

\medskip

In order to prove Theorem~\ref{main-theorem}, we now introduce the following auxiliary notation.

  \begin{definition}
		Let $l>1$ and assume that the diameter of $f$ is $k$.
		Then $f$ is called a \emph{potential $(k, n, l)$-lifting} if $\lvert F_{(m)}^{-1}(y) \rvert \leq l \cdot 2^{k-1}$ for any $y \in \F_2^{m-k+1}$ for every $m$ such that $k \leq m \leq n$.
	\end{definition}
	
	\begin{lemma}\label{almost-then-potential-l}
		If $f$ is an almost lifting, then $f$ is a potential $(k, n, \ell(f))$-lifting for all $n \geq k$.
	\end{lemma}
	
	\begin{proof}
		If $\ell(f)<\infty$,
		then $\lvert f^{-1}(y)\rvert \leq\ell(f)$ for all $y\in\F_2^n$,
		so $\lvert f^{-1}(Y)\rvert \leq\ell(f)\lvert Y\rvert$ for any $Y\subseteq\F_2^n$.
		Let $k\leq m\leq n$, pick $y\in\F_2^{m-k+1}$ and define $Y$ as in the proof of Lemma~\ref{uniform-distribution}, then
		\[
		\lvert F_{(m)}^{-1}(y)\rvert=\frac{\lvert F^{-1}(Y)\rvert}{2^{n-m}} \leq \frac{\ell(f) \lvert Y\rvert}{2^{n-m}}=\frac{\ell(f) 2^{n-(m-k+1)}}{2^{n-m}}=\ell(f)2^{k-1}.
		\]
	\end{proof}
	
	\begin{lemma}
		Fix some $l>1$ and let $m\geq k$.
		If $F_{(m)}$ is not balanced, then for any sufficiently large $r$, there exists $z \in \F_2^{rm-k+1}$ such that $\lvert F_{(rm)}^{-1}(z)\rvert>l\cdot 2^{k-1}$.
	\end{lemma}
	
	\begin{proof}
		If $F_{(m)}$ is not balanced,
		there exist $y \in \F_2^{m-k+1}$ and a rational number $c>1$ such that $\lvert F_{(m)}^{-1}(y) \rvert = c \cdot 2^{k-1}$.
		Let $X = \{ x \in \F_2^m \colon F_{(m)}(x) = y\}$ and choose a natural number $r$ so such that $c^r>l$. Then $F_{(rm)}$ maps any element of the form $(x_1, \dotsc, x_r)$ with $x_i \in X$ to an element of the form $(y, y_1, y, y_2, \dotsc, y_{r-1}, y)$ with $y_i \in \F_2^{k-1}$ for $1\leq i\leq r-1$. Let
		\[
		\begin{split}
			X^r &= \{ (x_1, \dotsc, x_r) \in\F_2^{rm} : x_i \in X\}, \\
			Z &= \{ (y, y_1, y, y_2, \dotsc, y_{r-1}, y) \in\F_2^{rm-k+1} : y_i \in \F_2^{k-1} \}.
		\end{split}
		\]
		Then $F_{(rm)}$ maps $X^r$ onto $Z$.
		We see that $\lvert X^r\rvert = c^r 2^{r(k-1)}$ and $\lvert Z\rvert = 2^{(r-1)(k-1)}$, and it follows that there exists one element $z\in Z$ such that the size of its inverse image is at least $c^r 2^{k-1}$.
	\end{proof}
	
	\begin{corollary}\label{potential-l}
		Let $l>1$.
		Assume that $f$ is a potential $(k, n, l)$-lifting for all $n\geq k$. Then $f$ is a potential $(k, n)$-lifting for all $n\geq k$.
	\end{corollary}
	
	\begin{proof}
		Assume that there is some $n$ such that $f$ is not a potential $(k, n)$-lifting.
		Then there exists $m$ with $k\leq m\leq n$ such that $F_{(m)}$ is not balanced, and by the above lemma, there exists $m'$ and $z \in \F_2^{m'-k+1}$ such that $\lvert F_{(m')}^{-1}(z)\rvert>l\cdot 2^{k-1}$.
		For $n'\geq m'$, it follows that $f$ is not a potential $(k, n',l)$-lifting.
	\end{proof}
	
	\begin{remark}
		Pick some $l>1$.
		Let $S_{k, n, l}$ denote the set of all Boolean $f\colon \F_2^k\to\F_2$ such that $f$ is a potential $(k, n, l)$-lifting and $f(0, 0, \dotsc , 0) = 0$,
		and let $S_{k,l} = \{ f\colon \F_2^k\to\F_2 \mid f \in S_{k, n, l} \text{ for all } n\geq k\}$. Then, we have
		\[
		\lvert S_{k, l} \rvert =
		\lim_{n \rightarrow \infty} \lvert S_{k, n, l} \rvert =
		\lim_{n \rightarrow \infty} \lvert S_{k, n, 1} \rvert =
		\lvert S_k \rvert.
		\]
		Note that the limits exist since the number of potential $(k,n)$-liftings is bounded from above by $2^{2^k}$ and decreases with growing $n$.
	\end{remark}
	
	\begin{theorem}\label{main-theorem}
		Let $f\colon \F_2^k\to\F_2$.
		Then $f$ is a potential $(k,n)$-lifting for all $n\geq k$ if and only if $f$ is an almost lifting.
		
		Moreover, if $f$ is an almost lifting, then $\ell(f)\leq 2^{k-1}$.
	\end{theorem}
	
	\begin{proof}
		First, assume $f$ is a potential $(k,n)$-lifting for every $n\geq k$.
        Pick any $n\geq k$ and consider the map
		\[
		F_{(n+k-1)}\colon \F_2^{n+k-1} \to \F_2^n.
		\]
		For every $y\in\F_2^n$, we have that
		\[
		\lvert F^{-1}(y)\rvert \leq \lvert F_{(n+k-1)}^{-1}(y)\rvert = 2^{k-1}.
		\]
		Thus, $f$ is an almost lifting.
		
		On the other hand, Lemma~\ref{almost-then-potential-l} in combination with Corollary~\ref{potential-l} implies that an almost lifting is a potential $(k, n)$-lifting for all $n \geq k$.
		
		For the second statement, we note that $S_k=S_{k,1}$, and if $f$ is a potential $(k,n)$-lifting for all $n\geq k$, then $\ell(f)=2^{k-1}$.
	\end{proof}

	\section{Surjective cellular automata}
	
	Let $P_n$ be the set of $n$-periodic doubly infinite (i.e., indexed by $\Z$) bit strings and let $P$ be the set of all periodic doubly infinite bit strings, i.e., $P=\cup_{n\geq 1}P_n$.

    The space $\F_2^\Z=\prod_{i=-\infty}^{\infty}\F_2$ with the product topology is a so-called Cantor space, and in particular, it is compact and metric, and contains $P$ as a dense subspace. A function $F\colon\F_2^\Z\to\F_2^\Z$ is called a cellular automaton if it is continuous and shift-invariant. Clearly, a cellular automaton $F$ restricts to a shift-invariant map $P_n\to P_n$ for all $n\geq 1$, and to a shift-invariant continuous map $P\to P$. Moreover, $F$ is called reversible if there exists a cellular automata $G$ such that $F\circ G=G\circ F=I$. It is known that a cellular automaton is reversible if and only if it is bijective \cite{hedlund}.
	
	Let $f$ be Boolean function of diameter $k$, $w\in\Z$, and let $F\colon\F_2^\Z\to\F_2^\Z$ be the map 
	\[
	F(x)_{i+w} = f(x_i,x_{i+1},\dotsc,x_{i+k-1}),
	\]
	that is, cell~$i+w$ of the state after $F$ is applied depends on the $k$-cells $i,i+1,\dotsc,i+k-1$ of the previous state. Then $F$ is a cellular automaton and every cellular automaton $F$ is defined by such a local rule $f$. If $w$ is nonzero, we can replace $F$ by $F\circ \sigma^{-w}$, so it suffices to study the case $w=0$.

\begin{remark}
    	Let $f\colon\F_2^k \to \F_2$ be a Boolean function of diameter $k$.
		Then the following are equivalent \cite[Theorem~7]{kari-survey}:
		\begin{itemize}
            \item[(i)] $F\colon P\to P$ is injective,
			\item[(ii)] $F\colon \F_2^\Z\to \F_2^\Z$ is injective,
   			\item[(iii)] $f$ is a proper lifting.
		\end{itemize}
\end{remark}
 
	\begin{theorem}
		Let $f\colon\F_2^k \to \F_2$ be a Boolean function of diameter $k$.
		Then the following are equivalent:
		\begin{itemize}
			\item[(i)] $F\colon P\to P$ is surjective,
			\item[(ii)] $F\colon \F_2^\Z\to \F_2^\Z$ is surjective,
			\item[(iii)] $F_{(m)}\colon \F_2^m\to \F_2^{m-k+1}$ is surjective for all $m\geq k$,
			\item[(iv)] $f$ is an almost lifting.
		\end{itemize}
	\end{theorem}
	
	\begin{proof}
		By Theorem~\ref{main-theorem}, $f$ is an almost lifting if and only if $f$ is a potential $(k, n)$-lifting for any $n \geq k$, which is equivalent to $F_{(m)}$ being balanced for any $m$, $m \geq k$.
		
		(i)~$\Longrightarrow$~(ii):
		Since $\F_2^\Z$ is compact and $P$ is dense in $\F_2^\Z$, if $F(\F_2^\Z)$ contains $P$ it must contain all of $\F_2^\Z$ (this is also explained in \cite[Theorem~5 and~6]{schoone-daemen}).
		
		(ii)~$\Longrightarrow$~(iii):
		Pick $y\in \F_2^{m-k+1}$, and expand it to an element of $y'\in\F_2^\Z$ by setting $y'_i = y_j$ for $i \equiv j\, (\operatorname{mod} m-k+1)$. Find $x'\in\F_2^\Z$ such that $F(x')=y'$, and define $x\in\F_2^m$ by $x_i=x'_i$. 
		Then $F_{(m)}(x)=y$.
		
		(iii)~$\Longrightarrow$~(iv):
		Suppose that $f$ is not an almost lifting. Then there exists some $m$, $m \geq k$ such that $F_{(m)}$ is not balanced. It follows that there exists some bitstring $y$ of length $m-k+1$ such that $\lvert F_{(m)}^{-1}(y) \rvert \leq 2^{k-1}-1$. For a positive integer $r$, let $S_r$ denote the set of bitstrings $y'$ of length $rm-k+1$ consisting of $y$, then any $k-1$ bits, then $y$, then any $k-1$ bits, and so on. There are $2^{(r-1)(k-1)}$ elements in $S_r$, but at most $\left(2^{k-1} - 1\right)^r$ elements in $\lvert F_{(rm)}^{-1}(S_r) \rvert$. Thus, if $r$ is large enough that $\left(1 - \frac{1}{2^{k-1}} \right)^r < \frac{1}{2^{k-1}}$, then $F_{(rm)}$ is not surjective.
		
		(iv)~$\Longrightarrow$~(i):
		Suppose $f$ is an almost lifting. Let $y$ be any finite bitstring of some length $n \geq k$, and let $m=2^k n+k-1$. Since $F_{(m)}$ is balanced, it is surjective, so there exists $x \in \F_2^m$ such that $F_{(m)}(x)=yy \dotsc y$. Note that $y$ is determined by a substring of $x$ of length $n+k-1$. Because there are only $2^k$ distinct strings of length $k$, there must exist $i, j \in \{0, n, \dotsc, 2^k n\}$ such that $i<j$ and $x_i = x_j, x_{i+1} = x_{j+1}, \dotsc, x_{i+k-1} = x_{j+k-1}$. Let $x' \in P$ of period $j-i$ be given by $x'_l = x_l$ for $i \leq l \leq j - 1$. Then we have
		$F(x')=\dotsc y y \dotsc$. Thus, $F\colon P \to P$ is surjective.
	\end{proof}
	
	We remark that some of the above could also be deduced from \cite[Section~5]{hedlund}.

	\section{Desirable properties for almost bijectivity}
	
	We would like to find Boolean functions that are non-bijective shift-invariant functions with preferably these properties for all $n$:
	\begin{itemize}
	\item[(P1)] the number of collisions $\max_y \lvert F^{-1}(y)\rvert$ should be low,
	\item[(P2)] (size of the image of $F$)/(size of the codomain of $F$) should be high,
	\item[(P3)] the image $\operatorname{Im}(F)$ and its complement should be unstructured in $\F_2^n$.
	\end{itemize}

	Moreover, to have applications in cryptography, almost bijective functions should otherwise have good properties with respect to differential uniformity, nonlinearity, algebraic degree, etc., that will be discussed in the next section.
	
	First, regarding (P1), we already know from the previous section that if $f$ is an almost lifting, then $\ell_n(f)\leq 2^{k-1}$ for all $n\geq k$. In this case, computer experiments suggest that the collision number pattern, that is, the sequences $\big(\ell_n(f)\big)_{n\geq k}$ are sometimes periodic and sometimes nonperiodic, and often take values a bit lower than $2^{k-1}$.
	
	\medskip
	
	To investigate (P2), we define the distribution of the sizes of preimages by letting $c_{j,n}(f)$, for $n\geq k$ and $j=0,1,2,\dotsc,2^n$, be given by
	\[
	c_{j,n}(f) = \Big\lvert \{ y\in\F_2^n : \lvert F^{-1}(y)\rvert = j \} \Big\rvert,
	\]
	and one would typically say the distribution is good if $c_{0,n}(f)$ is small and $c_{1,n}(f)$ is large, relative to $2^n$, which again should mean that all $c_{j,n}(f)$ for $j\geq 2$ are small.
	
	Moreover, note that we have
	\[
	\frac{2^n}{2^n-c_{0,n}(f)} \leq \ell_n(f) \leq c_{0,n}(f) + 1
	\quad\text{ and }\quad
	\ell_n(f) -1 \leq c_{0,n}(f) \leq 2^n\left(1-\frac{1}{\ell_n(f)}\right).
	\]
	These are derived from considering the extreme cases with either only one instance or $2^n-c_{0,n}(f)$ instances of $\lvert F^{-1}(y)\rvert = \ell_n(f)$, and $c_{0,n}(f)$ instances of $\lvert F^{-1}(y)\rvert = 0$ and $\lvert F^{-1}(y)\rvert = 1$ otherwise.
	
	Given $f$, let $\iota(n)=\Big\lvert \{ y\in\F_2^n : \lvert F^{-1}(y)\rvert = 0 \} \Big\rvert=c_{0,n}(f)$. Clearly, if $f$ is a proper $(k, n)$-lifting, then $\iota(n)=0$. We are interested in functions $f$ for which $\iota(n)$ is not identically 0, but is bounded by some slowly growing function.
	
	\begin{proposition}\label{virtual}
		Given a positive integer $d$, let $f\colon\F_2^{d+1}\to\F_2$ be the function of algebraic degree $d$, given by $f(x_1,\dotsc,x_{d+1}) = x_1 \oplus x_2 \dotsm x_d (x_{d+1} \oplus 1)$. Then, for $n > d$,
		
		\begin{equation}\label{eq:virtual}
			\iota(n) = 
			\begin{cases}
				d \cdot 2^{\frac{n}{d}-1} & \text{if }d|n,\\
				0 & \text{otherwise}\\ 
			\end{cases} 
		\end{equation}
	\end{proposition}
	
	\begin{proof}
		Take any $y \in \F_2^n$ and set $Y=\{ i : y_i=0 \}=\{ \beta_1<\beta_2<\dotsb<\beta_{\lvert Y\rvert} \}$. If $\beta_{i+1}-\beta_i\equiv 0\, (\operatorname{mod} d)$ for all $1\leq i\leq\lvert Y\rvert$, where $\beta_{\lvert Y\rvert+1}:=n+\beta_1$, then we say that $y$ satisfies ($*$). In particular, we note that if $y$ satisfies ($*$), then $d$ must divide $n$.
		
		Suppose $1 \leq \alpha_1 < \dotsc < \alpha_j \leq n$ are integers such that conditions (i)-(iv) are satisfied. The indices are considered modulo $j$ such that $\alpha_{j+1}$ is $\alpha_1$, and the set $\{\alpha_i+1, \dotsc, \alpha_{i+1}-1\}$ for $i=j$ should be read as $\{\alpha_j+1, \dotsc, n\} \cup \{1, \dotsc, \alpha_1-1\}$.
		\begin{itemize}
			\item[(i)] $y_{\alpha_i} = 0$ for all $i$.
			\item[(ii)] For each $i$, there is at most one element $l \in \{\alpha_i+1, \dotsc, \alpha_{i+1}-1\}$ such that $y_l=0$.
			\item[(iii)] If $\alpha_i \equiv \alpha_{i+1}\, (\operatorname{mod} d$), then it is required that such an element $l$ exists.
			\item[(iv)] If there is indeed such an element $l$, then it is required that $l \equiv \alpha_{i+1}\, (\operatorname{mod} d$).
		\end{itemize}
		Then there exists $x \in \F_2^n$ such that $F(x)=y$. Indeed, we can start with $x=y$ and make the following modifications for each $i \in \{1, \dotsc, j\}$. If there is no $l \in \{\alpha_i+1, \dotsc, \alpha_{i+1}-1\}$ such that $y_l=0$, shift the values of $x_{\alpha_{i+1}-d}, x_{\alpha_{i+1}-2d}, \dotsc$ until the end of the interval $(\alpha_i, \alpha_{i+1})$ is reached. If there is such an $l$, stop at that index. Note that $y_{\alpha_i} = f(x_{\alpha_i}, \dotsc, x_{\alpha_i+d}) = x_{\alpha_i} = 0$ for all $i$.
		
		Suppose first that $y$ does not satisfy $(*)$.
		If $m \geq 1$, there exists an $i$ such that $d$ does not divide $\beta_{i+1}-\beta_i$. We will let $\beta_i$ be one of the $\alpha$'s. Now traverse the $\beta$'s backwards (i.e., consider $\beta_{i-1}, \beta_{i-2}, \dotsc$ in turn) in search of new $\alpha$'s. Whenever the current $\beta$ is not congruent to the last added $\alpha$, add it. Otherwise, skip it and add the next $\beta$ regardless. Because of the starting condition, there will not be a conflict when we come back to where we started. An example: $n=10$, $d=3$, $y = (0, 1, 1, 0, 1, 0, 0, 1, 0)$. We have $\{\beta_i \}= \{1, 4, 6, 7, 9\}$. Since 6 and 7 are not congruent modulo 3, we can let 6 be an $\alpha$. Going backwards, we add 4, since 4 and 6 are not congruent. Now we skip 1 because 1 and 4 are congruent modulo 3 and get 9, and then 7. So $\{\alpha_i\} = \{4, 6, 7, 9\}$.
		
		Assume next that $y$ satisfies $(*)$ and $\lvert Y\rvert$ is odd.
		Then $\beta_{i+1} - \beta_i \equiv 0\, (\operatorname{mod} d)$ for each $i$,
		so by condition~(iii) it follows that every other $\beta_i$ is an $\alpha_i$, i.e.,
		exactly half of the $\beta_i$'s is an $\alpha_i$,
		but this is not possible if $\lvert Y \rvert$ is odd.
		The number of such elements $y$ is equal to $d$ (the number of residue classes) times the number of subsets of $\{1, \dotsc, \frac{n}{d}\}$ with an odd number of elements, which is $2^{\frac{n}{d}-1}$, and this is now an upper bound for $\iota(n)$ when $d|n$, while $\iota(n)$ must be $0$ otherwise.
		
		Finally, we suppose that $y$ satisfies $(*)$ and $\lvert Y\rvert$ is even.
		The indices $i$ such that $y_i=0$ are congruent modulo $d$,
		and we usually get two distinct possibilities for $x$ using the same method: If $y_i=0$ for $i \in \{\beta_1, \dotsc, \beta_{2j}\}$, then we can take either $\alpha_i = \beta_{2i-1}$ for all $i$ or $\alpha_i = \beta_{2i}$ for all $i$. The exception is $y=(1, \dotsc, 1)$, for which $x$ can either be equal to $y$, or $x_i=0$ for all $i$ in any given residue class modulo $d$ and $x_i=1$ otherwise. Since every additional inverse in this case removes an inverse from the previous case, it follows that $\iota(n)$ is at least $d(2^{\frac{n}{d}-1}-1)+d=d \cdot 2^{\frac{n}{d}-1}$.
	\end{proof}

	The functions given in Proposition~\ref{virtual} generalize the $\chi$ function (up to elementary equivalence). Computer experiments indicate that the values for $c_{0,n}(f)$ given in \eqref{eq:virtual} are lowest possible for almost liftings that do not induce a bijection for all $n$. We have checked all functions up to $k=5$ and also some classes for $k=6$, for $n\leq 20$, and Proposition~\ref{virtual} shows that such functions exist for all these $k$'s. We therefore conjecture that this bound is indeed optimal, and make the following definition.
	
	\begin{definition}
		A nonlinear function $f$ of diameter $k$ and $\operatorname{deg}(f)=d<k$ is called a \emph{virtual lifting} if it satisfies condition \eqref{eq:virtual} for all $n\geq k$.
	\end{definition}
 
	A complete list of almost liftings for $k\leq 5$ satisfying \eqref{eq:virtual} for all $n\leq 20$ is given in Appendix~\ref{appendix:virtual}. We believe they are all virtual liftings. Moreover, a complete list of Boolean functions for $k\leq 5$ that induce bijections for all $n$ is given in Appendix~\ref{appendix:proper} (the proof will be given in a forthcoming paper). We call such functions proper liftings.
    
	\medskip
	
	Furthermore, for (P3) and structuredness of the image, one can look at properties such as balancedness, strict avalanche, and collision difference. Balancedness for a given $n$ is defined by
	\[
	\max_{v\neq 0} \lvert\sum_{x\in\F_2^n}(-1)^{v\cdot F(x)}\rvert,
	\]
	and is $0$ if $F$ is bijective, and otherwise says something about how the outputs may accumulate around certain vectors.
	The strict avalanche criterion (the effect of changing one input; the best is if it flips half of the outputs) is given for each $v\neq 0$ and $1\leq i\leq n$ by setting $(v\cdot F)_i(x)=(v\cdot F)(x)\oplus(v\cdot F)(x\oplus e_i)$ and then compute
	\[
	\max_{1\leq i\leq n,\, v\neq 0} \lvert\sum_{x\in\F_2^n}(-1)^{(v\cdot F)_i(x)}\rvert.
	\]
    Finally, we would like the probability of differentials that imply a collision to be small, so we define the collision difference as
    \[
    \max_{v\neq 0}\lvert\{ x\in \F_2^n : F(x)=F(x\oplus v)\}\rvert.
    \]
 
	We think that balancedness, strict avalanche, and collision difference play a role for non-bijective functions, since we would like them to ``spread things out'' as most as possible.
	
	\section{Desirable cryptographic properties}
	
	Good cryptographic properties generally include aspects such as algebraic degree, nonlinearity, differential uniformity, and differential branch number. It is also desirable that the Boolean function has a fairly simple polynomial expression, to achieve low computational complexity.
	
	
	\medskip
	
	First, the differential probability of $F$ is defined for $a,b\in\F_2^n$ by
	\[
	\operatorname{DP}(a,b)=\frac{1}{2^n}\lvert\{ x\in\F_2^n : F(x\oplus a)\oplus F(x)=b \}\rvert.
	\]
	The differential probability uniformity $\operatorname{DPU}(F)$ is then $\max\{ \operatorname{DP}(a,b) : a,b\in\F_2^n,\, a\neq 0 \}$ and we want this to be low. The differential uniformity $\operatorname{DU}(F)$ is $2^n$ times this.

\begin{lemma}
    The differential uniformity of $F$ does not depend on the linear terms of $f$. Moreover, if $m$ is the number of variables that $f$ depends upon nonlinearly, then the differential uniformity of $F$ is at least $2^{n-m}$.
\end{lemma}

\begin{proof}
    If we replace $f(x)$ by $f(x) \oplus x_j$ for some $j$, then $F(x \oplus y) \oplus F(y)$ is replaced by $F(x \oplus y) \oplus F(y) \oplus \left(x_j, x_{j + 1}, \dotsc, x_{j - 1} \right)$, and we get the same value for the differential uniformity. In other words, the differential uniformity is independent of linear terms.

    If $f$ depends nonlinearly on $m$ variables, then $n-m$ of the coordinate functions of $F$ depend linearly of, say $x_n$. Therefore, $n-m$ of the coordinate functions of $F(x\oplus (0,\dotsc,0,1))\oplus F(x)$ are $0$ or $1$, constantly, so $F(x\oplus (0,\dotsc,0,1))\oplus F(x)$ takes at most $2^m$ values. That is, one of the values must be reached at least $2^{n-m}$ times, so the differential uniformity must be at least $2^{n-m}$.   
\end{proof}

\begin{remark}\label{apn-lifting}
        There are many functions for which the bound $2^{n-k}$ is achieved, even for $n=k+1$. E.g., if $f(x)=x_1x_2$, then the differential uniformity of the induced function is $2$ for $n=2$ and $2^{n-2}$ for $n \geq 3$ (checked up to $n = 10$).
        
		On the other hand, we have verified with computer assistance that among the almost liftings of diameter $k=3,4,5$, the lowest possible differential uniformity appears to be $2^{n-k+1}$ for all $n\geq k$. Recall that a vectorial Boolean function $F\colon\F_2^n \to \F_2^n$ is called almost perfect nonlinear (APN) if the differential uniformity of $F$ is $2$.
\end{remark}

        In light of the above, we say that a Boolean function $f$ is an \emph{APN~lifting} if for every $n\geq k$, the induced version of $f$ to $F\colon\F_2^n\to\F_2^n$ has differential uniformity at most $2^{n-k+1}$ (in particular, $F$ is an APN~function for $n=k$).
  
        Suppose $f$ is an APN lifting. If $f$ is permutive, on the form $x_1 \oplus g(x_2, \dotsc, x_k)$, say, then $f$ is an almost lifting by Lemma~\ref{permutive}. Therefore, a single permutive APN lifting in general gives $2^{k-2}$ elementary equivalence classes of APN liftings. If $r(x_2, \dotsc, x_k) \neq (x_2, \dotsc, x_k)$, then $x_1 \oplus g \circ r (x_2, \dotsc, x_k)$ is also an APN lifting, giving a total of $2^{k-1}$ elementary equivalence classes. Thus, the $2$, $8$ and $16$ elementary equivalence classes, respectively of APN liftings that we found among all almost liftings for $k=3, 4, 5$ come from adding linear terms to one single permutive APN lifting for each $k$, namely $x_1 \oplus x_2 x_3$, $x_1 \oplus x_2 (x_3 \oplus x_4)$ and $x_1 \oplus x_2 (x_3 \oplus x_4 \oplus x_5) \oplus x_3 x_5$. Note that each one is of algebraic degree 2 (for more on shift-invariant APN~functions, see \cite[Section~4.2]{carlet-APNRS}).

        We have also searched through all almost liftings of degree 2 for $k=6$, and in this case the function $x_1 \oplus x_2 (x_3 \oplus x_5) \oplus x_3 (x_4 \oplus x_5 \oplus x_6) \oplus x_4 (x_5 \oplus x_6)$ yields the 32 elementary equivalence classes that were found.

        Finally, we searched through all permutive almost liftings of degree 2 for $k=7$, and in this case there are 640 elementary equivalence classes, obtained by adding linear terms to essentially 10 different functions.

	\medskip
	
	The nonlinearity of a Boolean function $f$ is the minimum Hamming distance between $f$ and affine functions. We shall denote it by $\operatorname{nl}(f)$.
	To protect against certain linear attacks, the nonlinearity of $F$ is given by $\operatorname{NL}(F)=\min_{v\neq 0}\operatorname{nl}(v\cdot F)$, and we have (see \cite[Definition~29]{carlet-book})
	\[
	2\operatorname{NL}(F)=2^n-\max_{a,b,b\neq 0} \Big\lvert\sum_{x\in\F_2^n}(-1)^{a\cdot x+b\cdot F(x)}\Big\rvert.
	\]
    
	Define the correlation for $a,b\in\F_2^n$ by
	\[
	\operatorname{C}(a,b)=\frac{1}{2^n}\sum_{x\in\F_2^n}(-1)^{a\cdot x + b\cdot F(x)}
	\]
	and the linear potential of a linear approximation $(a,b)$ by $\operatorname{LP}(a,b)=\operatorname{C}(a,b)^2$.
	The relationship between correlation and nonlinearity is therefore
	\[
	2\operatorname{NL}(F)+2^n\max_{b\neq 0}\sqrt{\operatorname{LP}(a,b)}=2^n.
	\]
	The linear potential uniformity is then
	\[
	\operatorname{LPU}(F)=\max_{a,b,b\neq 0}\operatorname{LP}(a,b) = \left( 1 - \frac{\operatorname{NL}(F)}{2^{n-1}} \right)^2.
	\]
	
\begin{remark}
In general, we have
\[
\operatorname{NL}(F) \leq 2^{n-k}\operatorname{nl}(f) \leq 2^{n-1}-2^{n-(k+2)/2},
\]
or equivalently that $\operatorname{LPU}(F)\geq 2^{-k}$, where equality can only be obtained for $k$ even and when $f$ is a bent function. Computer experiments suggest that whereas equality is achieved for every $n$ and even $k\leq n$, the function $f(x)=x_1x_2$ modulo linear terms for $k=2$ might be the only function achieving the bound for every $n$.

From \cite[Theorem~4]{charpin-peng} it follows that if $F$ is differentially two-valued, then
\begin{equation}\label{nl-du}
\operatorname{NL}(F) \leq 2^{n-1} - (2^{n-2}\operatorname{DU}(F))^{1/2},
\end{equation}
so if $\operatorname{NL}(F) \geq 2^{n-1}-2^{n-(k+2-j)/2}$ for some integer $j\geq 0$, then $\operatorname{DU}(F)\leq 2^{n-k+j}$. Note that \eqref{nl-du} is equivalent to $\operatorname{LPU}(F)\geq \operatorname{DPU}(F)$. Moreover, \cite[Theorem~6]{charpin-peng} implies that if $F$ is also quadratic, $n$ is even, and $\operatorname{DU}(F)=2^{n-k+j}$ for some integer $j\geq 1$ such that $k+j$ is odd, then $\operatorname{LPU}(F)\geq 2\operatorname{DPU}(F)$. It remains to be studied whether one can give a meaningful definition of a ``highly nonlinear lifting''.

\end{remark}
	
	Furthermore, the algebraic degree of $F$ is given by $\operatorname{deg}(F)=\max_{v\neq 0}\operatorname{deg}(v\cdot F)$. When $F$ is shift-invariant, this is the same as the algebraic degree of $f_1$. Indeed, for all $1\leq i\leq n$ we clearly have $\operatorname{deg}(f_1)=\operatorname{deg}(f_i)$ and $\operatorname{deg}(F)\geq\operatorname{deg}(f_1)=\operatorname{deg}(f_i)$. Moreover, since $v\cdot F$ are sums of the $f_i$'s we must have $\operatorname{deg}(F)\leq\operatorname{deg}(f_i)$.
		
	\medskip
	
	Finally, we consider the differential branch number, which is given by
	\[
	\min_{x\neq y}\{ \operatorname{wt}(x \oplus y) + \operatorname{wt}(F(x) \oplus F(y)) \},
	\]
	as a measure of how effectively the S-box spreads input differences into output differences.
    
	\section{Selected candidates}\label{sec-selection}
	
	After extensive searching, we now consider a few candidates more closely:
	\begin{itemize}
		\item[(A1)] $f(x)=x_1\oplus x_2(x_3\oplus 1)$
		\item[(A2)] $f(x)=x_1\oplus x_2x_3$
		\item[(B1)] $f(x)=x_1\oplus x_2(x_3\oplus x_4)$
		\item[(B2)] $f(x)=x_1\oplus x_2(x_3\oplus x_4\oplus 1)$
		\item[(B3)] $f(x)=x_1\oplus x_4(x_2\oplus x_3\oplus 1)$
		\item[(C1)] $f(x)=x_2\oplus x_3\oplus x_4(x_1\oplus x_2)(x_3\oplus 1)$
		\item[(C2)] $f(x)=x_1\oplus x_4\oplus x_3(x_2\oplus x_4\oplus x_2x_4)$
		\item[(D1)] $f(x)=x_2 \oplus x_3 ((x_1 \oplus x_2) (x_4 \oplus 1) \oplus x_4 x_5 \oplus 1)$
		\item[(D2)] $f(x)=x_2 \oplus x_3 (x_1 \oplus 1) \oplus x_4 ((x_2 \oplus 1) (x_5 \oplus 1) \oplus x_3 (x_1 \oplus x_5))$
		\item[(D3)] $f(x)=x_2 \oplus x_4 (x_5 \oplus 1) (x_1 \oplus x_3)$
		\item[(E1)] $f(x)=x_2 \oplus x_1 (x_4 (x_3 \oplus 1) \oplus (x_4 \oplus 1) x_5 (x_2 \oplus x_3 \oplus 1))$
	\end{itemize}
	
	\medskip
	
	
	There are (up to elementary equivalences) eight quadratic functions of diameter~$4$ that have differential probability uniformity $\frac{1}{8}$ and linear potential uniformity $\frac{1}{4}$ (checked for $n\leq 10$), six of them having nonperiodic collision number pattern, and we have picked three functions in the (B)~class with sparse ANF, where (B2) and (B3) have a nonperiodic pattern. The only other of the above functions with nonperiodic collision number pattern is (C2). The functions in the (D)~class are virtual liftings, while (E1) is a proper lifting.
	For (A1) and (A2) the differential probability uniformity is $\frac{1}{4}$ for every $n$ that we checked. 
	
	
	Here is a summary of our computations, but be aware that our values for $\operatorname{DPU}$ and $\operatorname{LPU}$ are only checked for $n\leq 9$ or $10$. For some functions the value is indeed constant for each $n\leq 9$, while for other functions there are minor fluctuations around the values given in the table. It also looks like the (P2) values stabilize when $n$ grows for the (B) and (C) functions. For the (A), (D), and (E) functions the values are sometimes (periodically) the ones given, and otherwise $1$. 
	\[
	\begin{array}{|c|c|c|c|c|c|} \hline
		& k & \deg & \operatorname{DPU} & \operatorname{LPU} & \text{(P2) for $n=10$} \\ \hline
		\text{(A1)} & 3 & 2 & 1/4 & 1/4 & .97 \\
		\text{(A2)} & 3 & 2 & 1/4 & 1/4 & .87 \\
		\text{(B1)} & 4 & 2 & 1/8 & 1/4 & .84 \\
		\text{(B2)} & 4 & 2 & 1/8 & 1/4 & .86 \\
		\text{(B3)} & 4 & 2 & 1/8 & 1/4 & .83 \\
		\text{(C1)} & 4 & 3 & 5/16 & 9/16 & .90 \\
		\text{(C2)} & 4 & 3 & 5/16 & 9/16 & .71 \\
		\text{(D1)} & 5 & 3 & 7/32 & 1/4 & .95 \\
		\text{(D2)} & 5 & 3 & 7/32 & 1/4 & .95 \\
		\text{(D3)} & 5 & 3 & 9/32 & 9/16 & .95 \\
		\text{(E1)} & 5 & 4 & 1/4 & 25/64 & 1 \\ \hline
	\end{array}
	\]

    The values for balancedness of (A1) is $2^{n/2+1}$ when $n$ is even, and for (A2) and the three (B) functions it also seems to be approximately $2^{n/2+1}$ for all $n$. For the (D) functions we get $3\cdot 2^{n/3}$ when $n$ is a multiple of $3$. Very rough estimates for (C1) and (C2) are $2^{0.8n}$ and $2^{0.6n}$, respectively. Strong avalanche seems to be approximately $2^{n-3}$, respectively, for both (A1) and the three (B) functions, when $n$ grows and is even.
    
    All the above functions have differential branch number $2$, except (C2), that has $3$.
		
	The final property that we consider is the collision difference, i.e., we would like $\max_{a\neq 0} \operatorname{DP}(a,0)$ to be small. For (A1) this is $2^{-n/2}$ when $n$ is even, and for (A2) and the three (B) functions it is approximately $2^{-2n/3}$ for all $n$.

\medskip

Higher degree and denser ANF improves resistance to certain differential and integral attacks, respectively, but is more expensive. Moreover, the periodic collision pattern are easier to reason about in mode designs. If one wants symmetry and low implementation cost, the (B1) function seems like a good choice, but any (B)~type function is of interest. If non-quadratic is preferable, then (D1) or (D2) for decent cryptographic properties, or (D3) for a more compact description, look like promising options.

\section{Conclusion and further research}

This paper contributes to the ongoing development of efficient cryptographic primitives based on near-permutation nonlinear layers and present new directions of designing S-boxes defined by compact rules using symmetries and cellular automata theory. More specifically, we investigate the cryptographic potential of almost bijective shift-invariant vectorial Boolean functions. 

We explicitly describe the class of Boolean functions, called almost liftings, that we envision as a tool for constructing such S-boxes, and we discuss several examples with good cryptographic properties induced from these functions. Even though the $\chi$ function already has good properties, the advantage of looking at the larger class of almost liftings is to get a broader variety of properties to benefit from, depending on applications. Knowledge about this wider range of possibilities will be significant for design of, e.g., lightweight cryptography. An immediate next step is then co-designing linear layers to compensate sparse ANFs when desired.

For future research, we are going to pursue new theoretical results on differential uniformity and nonlinearity, and tightening image-size estimates for shift-invariant transformations. Further, there is more work to do in computer searching for almost liftings of larger diameters that induce non-bijective S-boxes with good cryptographic properties, and finding new applications in ``near-permutation-based cryptography''.



Finally, the concept of almost liftings, and the equivalence with surjective cellular automata, can be extended to $\F_p^k\to\F_p$ for all fields of characteristic $p>2$, i.e., one gets that $\sup_n\ell_n(f)\leq p^{k-1}$ or it is infinite. In fact, since there is no algebra involved, we may as well look at functions $\{0,1,\dotsc, p-1\}^k\to\{0,1,\dotsc, p-1\}$ for any integer $p>2$. Grassi discusses non-bijective shift-invariant transformations over odd prime fields \cite{Grassi}.

\newpage
    
	\appendix
	
	\section{Counting the number of liftings}\label{liftings}
	
	Tables for $k=3,4,5$ and in part for $k=6$ (number of elementary equivalence classes)

    $k=3$
	\[
	\begin{array}{|c|c|c|c|c|c|} \hline
		n & \#\,\text{potential} & f(0) \neq f(1) & \#\,\text{liftings} & \deg = 1 & \deg = 2 \\ \hline
		3 & 13 & 8 & 6 & 0 & 6 \\
		4 & 5 & 3 & 1 & 1 & 0 \\
		5 & 4 & 2 & 2 & 1 & 1 \\
		6 & 4 & 2 & 0 & 0 & 0 \\
		7 & 4 & 2 & 2 & 1 & 1 \\
		8 & 4 & 2 & 1 & 1 & 0 \\
		9 & 4 & 2 & 1 & 0 & 1 \\
		10 & 4 & 2 & 1 & 1 & 0 \\
		11 & 4 & 2 & 2 & 1 & 1 \\
		12 & 4 & 2 & 0 & 0 & 0 \\
		13 & 4 & 2 & 2 & 1 & 1 \\
		14 & 4 & 2 & 1 & 1 & 0 \\
		15 & 4 & 2 & 1 & 0 & 1 \\
		16 & 4 & 2 & 1 & 1 & 0 \\
		17 & 4 & 2 & 2 & 1 & 1 \\
		18 & 4 & 2 & 0 & 0 & 0 \\
		19 & 4 & 2 & 2 & 1 & 1 \\ \hline
	\end{array}
	\]

    $k=4$
	\[
	\begin{array}{|c|c|c|c|c|c|c|} \hline
		n & \#\,\text{potential} & f(0) \neq f(1) & \#\,\text{liftings} & \deg = 1 & \deg = 2 & \deg = 3 \\ \hline
		4 & 1665 & 887 & 205 & 1 & 12 & 192 \\
		5 & 536 & 281 & 59 & 1 & 6 & 52 \\
		6 & 124 & 64 & 6 & 1 & 3 & 2 \\
		7 & 77 & 39 & 4 & 0 & 0 & 4 \\
		8 & 73 & 36 & 4 & 1 & 0 & 3 \\
		9 & 73 & 36 & 3 & 1 & 0 & 2 \\
		10 & 73 & 36 & 4 & 1 & 0 & 3 \\
		11 & 73 & 36 & 5 & 1 & 0 & 4 \\
		12 & 73 & 36 & 2 & 1 & 0 & 1 \\
		13 & 73 & 36 & 5 & 1 & 0 & 4 \\
		14 & 73 & 36 & 3 & 0 & 0 & 3 \\
		15 & 73 & 36 & 3 & 1 & 0 & 2 \\
		16 & 73 & 36 & 4 & 1 & 0 & 3 \\
		17 & 73 & 36 & 5 & 1 & 0 & 4 \\
		18 & 73 & 36 & 2 & 1 & 0 & 1 \\
		19 & 73 & 36 & 5 & 1 & 0 & 4 \\
		20 & 73 & 36 & 4 & 1 & 0 & 3 \\
		21 & 73 & 36 & 2 & 0 & 0 & 2 \\
		22 & 73 & 36 & 4 & 1 & 0 & 3 \\
		23 & 73 & 36 & 5 & 1 & 0 & 4 \\ \hline
	\end{array}
	\]

    $k=5$
	\[
	\begin{array}{|c|c|c|c|c|c|c|c|} \hline
		n & \#\,\text{potential} & f(0) \neq f(1) & \#\,\text{liftings} & \deg = 1 & \deg = 2 & \deg = 3 & \deg = 4 \\ \hline
		5 & 75165111 & 38800984 & 2815556 & 2 & 483 & 89583 & 2725488 \\
		6 & & & 13316 & 2 & 117 & 731 & 12466 \\
		7 & & & 462 & 3 & 20 & 90 & 349 \\
		8 & 36080 & 18072 & 31 & 3 & 0 & 11 & 17 \\
		9 & 18808 & 9369 & 52 & 2 & 3 & 18 & 29 \\
		10 & 17921 & 8953 & 34 & 2 & 1 & 11 & 20 \\
		11 & 17885 & 8940 & 78 & 3 & 3 & 28 & 44 \\
		12 & 17882 & 8937 & 8 & 2 & 0 & 0 & 6 \\
		13 & 17881 & 8936 & 78 & 3 & 3 & 27 & 45 \\
		14 & 17881 & 8936 & 33 & 3 & 1 & 10 & 19 \\
		15 & 17881 & 8936 & 43 & 0 & 1 & 16 & 26 \\
		16 & 17881 & 8936 & 27 & 3 & 0 & 9 & 15 \\
		17 & 17881 & 8936 & 75 & 3 & 3 & 26 & 43 \\
		18 & 17881 & 8936 & 14 & 2 & 1 & 1 & 10 \\
		19 & 17881 & 8936 & 74 & 3 & 3 & 26 & 42 \\
		20 & 17881 & 8936 & 25 & 2 & 0 & 9 & 14 \\ \hline
	\end{array}
	\]

    $k=6$, deg $\leq 2$
	\[
	\begin{array}{|c|c|c|c|c|c|} \hline
		n & \#\,\text{potential} & f(0) \neq f(1) & \#\,\text{liftings} & \deg = 1 & \deg = 2 \\ \hline
		6 & 232090 & 119232 & 4850 & 3 & 4847 \\
		7 & 136330 & 69497 & 468 & 3 & 465 \\
		8 & 41462 & 22295 & 52 & 4 & 48 \\
		9 & 21784 & 11310 & 34 & 3 & 31 \\
		10 & 17078 & 8631 & 4 & 4 & 0 \\
		11 & 16701 & 8358 & 8 & 4 & 4 \\
		12 & 16593 & 8289 & 4 & 3 & 1 \\
		13 & 16581 & 8280 & 8 & 4 & 4 \\
		14 & 16579 & 8280 & 3 & 3 & 0 \\
		15 & 16579 & 8280 & 7 & 3 & 4 \\
		16 & 16579 & 8280 & 4 & 4 & 0 \\
		17 & 16579 & 8280 & 8 & 4 & 4 \\
		18 & 16579 & 8280 & 3 & 3 & 0 \\
        19 & 16579 & 8280 & 8 & 4 & 4 \\
        20 & 16579 & 8280 & 4 & 4 & 0 \\ \hline
	\end{array}
	\]

    A representative from each of the four classes of functions of degree 2 that are liftings for $n \in \{11, 13, 15, 17, 19\}$: $x_1 \oplus x_2 \oplus x_3 \oplus x_2x_3 \oplus x_2x_5 \oplus x_2x_6 \oplus x_3x_4 \oplus x_3x_5 \oplus x_4x_5 \oplus x_4x_6 \oplus x_5x_6$, $x_1 \oplus x_2 \oplus x_3 \oplus x_5 \oplus x_2x_3 \oplus x_4x_5 \oplus x_5x_6$, $x_1 \oplus x_3 \oplus x_4 \oplus x_5 \oplus x_2x_3 \oplus x_3x_4 \oplus x_5x_6$, $x_1 \oplus x_4 \oplus x_5 \oplus x_2x_3 \oplus x_2x_4 \oplus x_2x_6 \oplus x_3x_4 \oplus x_3x_5 \oplus x_3x_6 \oplus x_4x_5 \oplus x_5x_6$

	\section{List of virtual liftings}\label{appendix:virtual}
	
	In the tables in Appendix~B and~C, the given differentials are $2^n\operatorname{DPU}$ for $n=k,k+1,\dotsc,9$. In each case, the value for $n=10, 11, 12$ is $2^{n-9}$ times the value for $n=9$.
	
	In the $\ell_n(f)$ column of the first table, $a,b$ means that $\ell_n(f)=a$ if $n\in b\Z$ and is $\ell_n(f)=1$ otherwise.
	
	The twelve virtual liftings (up to elementary equivalence) for $k \leq 5$:
	\[
	\begin{array}{|c|c|c|c|c|c|} \hline
		k & \text{Boolean function} & \ell_n(f) & \deg & \operatorname{LPU} & \text{differentials} \\ \hline
		3 & x_1 \oplus x_2(x_3\oplus 1) & 3, 2 & 2 & 1/4 & 2, 4, 8, 16, 32, 64, 128 \\
		\hline
		4 & x_1 \oplus x_2 x_3 (x_4\oplus1) & 4, 3 & 3 & 9/16 & 6, 14, 28, 56, 112, 224 \\
		\hline
		4 & x_1 \oplus x_2 (x_3\oplus1) (x_4\oplus1) & 2, 3 & 3 & 9/16 & 6, 14, 28, 56, 112, 224 \\
		\hline
		5 & x_2 \oplus x_1 (x_3 x_4 \oplus x_5 (x_3 \oplus x_4 \oplus 1)) & 4, 3 & 3 & 9/16 & 10, 24, 42, 80, 162 \\
		\hline
		5 & x_2 \oplus x_3 ((x_1 \oplus x_2) (x_4 \oplus 1) \oplus x_4 x_5 \oplus 1) & 4, 3 & 3 & 1/4 & 8, 14, 28, 56, 112 \\
		\hline
		5 & x_2 \oplus x_3 (x_1 \oplus 1) \dotsc& 4, 3 & 3 & 1/4 & 8, 14, 28, 56, 112 \\
		& \dotsc \oplus x_4 ((x_2 \oplus 1) (x_5 \oplus 1) \oplus x_3 (x_1 \oplus x_5)) & & & & \\
		\hline
		5 & x_3 \oplus x_4 (x_5 (x_2 \oplus x_3 \oplus 1) \oplus 1) \dotsc & 4, 3 & 3 & 1 & 10, 18, 44, 84, 168 \\
		& \dotsc \oplus (x_4 \oplus 1) (x_2 \oplus x_3 (x_1 \oplus x_2)) & & & & \\
		\hline
		5 & x_2 \oplus x_4 (x_5 \oplus 1) (x_1 \oplus x_3) & 2, 3 & 3 & 9/16 & 12, 24, 34, 72, 144 \\
		\hline
		5 & x_1 \oplus x_2 x_3 x_4 (x_5 \oplus 1) & 5, 4 & 4 & 49/64 & 18, 38, 78, 156, 312 \\
		\hline
		5 & x_1 \oplus x_2 x_3 (x_4 \oplus 1) (x_5 \oplus 1) & 2, 4 & 4 & 49/64 & 22, 36, 74, 148, 296 \\
		\hline
		5 & x_1 \oplus x_2 (x_3 \oplus 1) (x_4 \oplus 1) (x_5 \oplus 1) & 2, 4 & 4 & 49/64 & 18, 38, 78, 156, 312 \\
		\hline
		5 & x_1 \oplus x_2 (x_3 \oplus 1) (x_4 (x_5 \oplus 1) \oplus 1) & 3, 4 & 4 & 25/64 & 14, 24, 48, 96, 192 \\
		\hline
	\end{array}
	\]

	\section{List of proper liftings}\label{appendix:proper}
	
	The six nonlinear Boolean functions of degree $\geq 2$ with $k \leq 5$ that are proper $(k,n)$-liftings for all $n \geq k$, up to elementary equivalence:
	\[
	\begin{array}{|c|c|c|c|c|} \hline
		k & \text{Boolean function}  & \deg & \operatorname{LPU} & \text{differentials} \\ \hline
		4 & x_2 \oplus x_1 (x_3 \oplus 1) x_4 & 3 & 9/16 & 6, 14, 30, 54, 108, 216 \\
		\hline
		5 & x_2 \oplus x_1 x_3 (x_4 \oplus 1) (x_5 \oplus 1) & 4 & 49/64 & 16, 34, 72, 148, 304 \\
		\hline
		5 & x_2 \oplus x_1 (x_3 \oplus 1) (x_4 \oplus 1) x_5 & 4 & 49/64 & 22, 34, 72, 146, 286 \\
		\hline
		5 & x_2 \oplus x_1 (x_4 (x_3 \oplus 1) \oplus (x_4 \oplus 1) x_5 (x_2 \oplus x_3 \oplus 1)) & 4 & 25/64 & 8, 18, 36, 68, 132\\
		\hline
		5 & x_3 \oplus x_1 x_2 (x_4 \oplus 1) x_5 & 4 & 49/64 & 18, 40, 78, 152, 300\\
		\hline
		5 & x_3 \oplus x_1 (x_2 \oplus 1) x_4 (x_5 \oplus 1) & 4 & 49/64 & 22, 50, 74, 148, 304 \\
		\hline
	\end{array}
	\]
	 Other classes of proper liftings are described in \cite{HO-2}.

	\bibliographystyle{plain}

\end{document}